\documentclass[letterpaper, 10pt, conference]{ieeeconf}
\IEEEoverridecommandlockouts                              

\usepackage[margin=1in]{geometry}
\usepackage[english]{babel}
\usepackage[utf8]{inputenc}

\usepackage{amsmath,amssymb,amsfonts}
\usepackage{dsfont}
\usepackage{graphicx}
\usepackage[colorlinks=true, allcolors=blue]{hyperref}
\usepackage{newtxtext, newtxmath}
\usepackage{cite}
\usepackage[table]{xcolor} 
\usepackage{array}
\usepackage{float}
\newcolumntype{P}[1]{>{\centering\arraybackslash}p{#1}}

\definecolor{highlight}{rgb}{0,0,0}

\usepackage{subcaption}

\title{A Behavioral approach to Constrained Linear Quadratic Control:
  Feasibility and Optimality}
\author{Soutrik Bandyopadhyay, Debasattam Pal and Shubhendu Bhasin
\thanks{Soutrik Bandyopadhyay and Shubhendu Bhasin are with the Department of
  Electrical Engineering, IIT Delhi and Debasattam Pal is with the Department of
  Electrical Engineering, IIT Bombay. 
{\tt\small \{Soutrik.Bandyopadhyay@ee.iitd.ac.in,} {\tt\small debasattam@ee.iitb.ac.in,
  sbhasin@ee.iitd.ac.in \}}}
}

\newcommand{\Bcons}{\mathfrak{B}_{\text{cons}}}
\newcommand{\Bsys}{\mathfrak{B}_{\text{sys}}}
\newcommand{\Baug}{\mathfrak{B}_{\text{aug}}}
\newcommand{\Badj}{\mathfrak{B}_{\text{adj}}}

\newcommand{\Baut}{\mathfrak{B}_{\text{aut}}}
\newcommand{\Bcontr}{\mathfrak{B}_{\text{contr}}}

\newcommand{\numIneq}{n_i}
\newcommand{\numEq}{n_s}
\newcommand{\numInit}{n_{ic}}
\newcommand{\dimW}{q}
\newcommand{\dimL}{m}
\newcommand{\dimLambda}{\numIneq}
\newcommand{\dimV}{p}

\newcommand{\leqp}{\preccurlyeq_{\mathcal{P}}}
\newcommand{\geqp}{\succcurlyeq_{\mathcal{P}}}

\newcommand{\inprod}[2]{\left\langle #1, #2 \right\rangle}

\newcommand{\poly}[1]{\mathbb{R}[#1]}
\newcommand{\polymat}[3]{\mathbb{R}[#1]^{#2 \times #3}}
\newcommand{\mat}[2]{\mathbb{R}[\eta,\eta^{-1}]^{#1 \times #2}}
\newcommand{\bilinearmat}[2]{\mathbb{R}[\xi,\eta]^{#1 \times #2}}

\newcommand{\lau}[1]{\mathbb{R}[#1,#1^{-1}]}

\let\oldforall\forall
\renewcommand{\forall}{\; \oldforall \, }

\let\oldexist\exists
\renewcommand{\exists}{\; \oldexist \, }

\newcommand{\sig}[1]{(\mathbb{R}^{#1})^{\mathbb{Z}}}

\newtheorem{thm}{Theorem}
\newtheorem{rmark}{Remark}
\newtheorem{prb}{Problem}
\newtheorem{dfn}{Definition}

\DeclareMathOperator{\kernel}{ker}
\DeclareMathOperator{\img}{im}

\begin{document}
\setlength{\abovedisplayskip}{4pt plus 3pt minus 2pt}
\setlength{\belowdisplayskip}{4pt plus 3pt minus 2pt}
\setlength{\abovedisplayshortskip}{0pt plus 3pt}
\setlength{\belowdisplayshortskip}{4pt plus 3pt minus 2pt}
\setlength{\dbltextfloatsep}{1.0\baselineskip plus 0.2\baselineskip minus 0.3\baselineskip}
\setlength{\dblfloatsep}{0.7\baselineskip plus 0.2\baselineskip minus 0.2\baselineskip}
\setlength{\textfloatsep}{1.0\baselineskip plus 0.2\baselineskip minus 0.4\baselineskip}
\maketitle

\begin{abstract}
This paper develops a constrained linear–quadratic (LQ) optimal control framework within the behavioral approach. Unlike classical formulations that rely on state–input representations, our method treats both the system dynamics and the constraints as trajectory-level objects. The system is modeled through a kernel representation, while inequality constraints are captured using behavioral inequalities defined over infinite-dimensional convex cones. Using quadratic differential forms (QDFs) to encode performance indices, we formulate the constrained LQ problem as a convex variational optimization over trajectories. By employing Gâteaux differentials, we derive the stationarity and complementarity conditions and show that they admit a representation as an augmented behavioral system whose unique decomposition into autonomous and controllable sub-behaviors yields a complete parameterization of all optimal constrained trajectories. The resulting conditions constitute a behavioral analogue of the Karush–Kuhn–Tucker (KKT) framework. A numerical example illustrates the effect of inequality constraints on the optimal behavior. The framework provides a general, trajectory-level representation for constrained optimal control and opens avenues for closed-loop and data-driven extensions. 

\end{abstract}

\section{Introduction}%
\label{sec:intro}

A fundamental way of analyzing dynamical systems is the behavioral approach,
where a system is defined as the set of trajectories following some underlying
relation among variables. Unlike input-output descriptions, the behavioral
approach does not presuppose a partition of variables into states, inputs, or
outputs; instead, it characterizes a system directly through the laws satisfied
by its measurable trajectories. Championed by Jan C. Willems
\cite{willems1991paradigms,willems1997introduction}
in the 1980s, behavioral systems theory has gained traction in recent years in
data-driven control, signal processing, robust control, and other areas (cf.
\cite{markovsky2021behavioral} and the references therein). 

Within this framework, optimal control problems can also be formulated without
resorting to state-space representations. A key tool is the theory of quadratic
differential forms (QDFs) \cite{willems1998quadratic}, which encode performance
functionals and Lyapunov-like expressions using two-variable polynomial
matrices, thereby generalizing the quadratic functionals of linear quadratic
control and making them amenable for study under the behavioral approach
\cite{willems2005optimality,willems2005linear}. This has enabled developments
in optimal control including LQ problems for multidimensional and
singular systems
\cite{pillai1999optimal,ferrante2000linear, qais2022continuous}.  

 While behavioral formulations for unconstrained LQ control are well understood,
 the incorporation of constraints remains comparatively underdeveloped,
 especially when constraints must themselves be enforced on trajectories rather
 than on pre-assigned input or state variables. Real-world control problems
 often involve unilateral or inequality-type restrictions, such as saturation
 limits, safety constraints, physical contact constraints, or sign constraints.
While classical constrained optimal control approaches have been reasonably
successful in the literature \cite{baotic2008efficient, borrelli2003constrained,
  nevistic1996constrained,
  pytlak1999numerical,
  brogliato2009some
}, such approaches typically rely on state-input models,
making them incompatible with the behavioral viewpoint. 

Alongside these model-based approaches, data-driven methods, most notably
Data-enabled predictive control
(DeePC) \cite{coulson2019data,baros2022online}
solve constrained predictive control problems directly from data, with
extensions addressing robustness \cite{berberich2020robust}, reachability
\cite{eising2019reachability}, and invariance \cite{bisoffi2020data}.
 However, all the approaches mentioned above assume a-priori partitioning of
 states, inputs and outputs and are typically expressed in the state-space representation.

 One promising approach in describing constraints in the behavioral setting can
 be traced back to A. A. ten Dam's thesis \cite{ten1997unilaterally}, which introduced
 Behavioral Inequalities, a geometric formulation of unilateral constraints on
 trajectories. By characterizing restricted behaviors via convex cones and
 contact conditions, this line of work provided early insight into how
 inequality-type dynamical constraints can be expressed without assigning roles
 to variables, including the contact problem for linear continuous-time
 dynamical systems \cite{ten1997contact}. The feasibility of such behavioral
 inequalities had been conjectured in \cite{ten2004farkas}, but the author
 claimed that such analysis is an open problem owing to the fact that matrices
 are generally non-invertible over the polynomial ring. In parallel, there has been recent interest in convex
 processes \cite{padoan2023data,eising2019reachability,eising2021data} for describing dynamical systems with inequality
 constraints. However, to the best of our knowledge, constrained optimal control
 has not yet been fully integrated with these behavioral constraint
 formulations.  

 \textcolor{highlight}{
In our recent work \cite{bandyopadhyay2025behavioral}, static behavioral inequalities were analyzed to establish geometric feasibility conditions for unconstrained trajectory sets. The present paper builds upon \cite{bandyopadhyay2025behavioral} by extending those algebraic inequality foundations into a dynamic optimal control setting. Specifically, while \cite{bandyopadhyay2025behavioral} focuses on static feasibility checks, this work formulates variational trajectory optimization under behavioral inequality constraints, derives stationarity and complementarity conditions, and constructs an augmented behavioral KKT system for constrained LQR synthesis.
}
In this paper, we bridge this gap by developing a constrained linear quadratic
optimal control framework within the behavioral approach. The system dynamics
and the constraints are both treated as behavioral objects: the system is
described by a kernel representation, while inequality constraints are modeled
via behavioral inequalities defined over infinite-dimensional convex cones. This
enables constraints to be imposed directly on trajectories, without
distinguishing between variables, and expressed through general polynomial
matrix relations. 

Our main contribution is to show that constrained LQ optimal control can be cast
as a convex variational problem over trajectories and solved using tools from
convex optimization in infinite-dimensional vector spaces. Leveraging the
structure of QDFs and the latent-variable image representation of controllable
behaviors, we derive the stationarity and complementarity conditions through
Gâteaux differentials and construct an augmented behavioral system encoding the
optimality conditions. We show that this augmented system admits a decomposition
into autonomous and controllable sub-behaviors, enabling explicit
parameterization of all optimal solutions. This yields a behavioral analogue of
the KKT conditions and a constructive path to computing optimal constrained
trajectories, while preserving the generality of constrained optimal control
problems.
\section{Preliminaries}%
\label{sec:prelims}
For a field $k$, we denote the set of all polynomials
in the indeterminates $x_1, \dots, x_{n}$ (polynomial ring) with coefficients in $k$ as
$k[x_1,x_2, \dots, x_n]$. For example $\poly{x_1,x_2, \dots, x_{n}}$ denotes
the polynomial ring with real coefficients. Correspondingly,
$\polymat{x_1, x_2, \dots, x_{n}}{g}{q}$ denotes the set of $g \times q$
matrices defined over the polynomial ring with real coefficients. We use
$\lau{\eta}$ to denote the Laurent polynomial ring with real coefficients.

A dynamical system in the context of behavioral approach is defined as
the tuple $\Sigma \triangleq (\mathcal{T}, \mathcal{W}, \mathfrak{B})$ where
$\mathcal{T}$ is the index set, $\mathcal{W}$ is the alphabet set, and the
$\mathfrak{B}$ denotes the subset of trajectories in
$\mathcal{W}^{\mathcal{T}} \triangleq \{w: \mathcal{T} \to  \mathcal{W}\}$ that follow an underlying
relationship between the variables \cite{willems1997introduction}.
For linear time-invariant systems, these laws are usually specified by polynomial
relations in the shift operator $\sigma(\cdot)$ (for the discrete-time case) or the
differential operator $\frac{d}{dt}(\cdot)$ (for the continuous-time case). For example
in the discrete-time,
\begin{equation}
  \label{eq:linear_behavior_continuous}
  \mathfrak{B} \triangleq \{w \in \sig{\dimW} : R(\sigma)w = 0\},
\end{equation}
where $R(\eta) \in \mat{g}{\dimW}$,  $\sigma(\cdot)$
denotes the shift operator defined as
 $(\sigma w)(k) \triangleq w(k+1) \forall k \in \mathbb{Z},$
with the inverse-shift operator
 $(\sigma^{-1} w)(k) \triangleq w(k-1) \forall k \in \mathbb{Z}.$ The representation in \eqref{eq:linear_behavior_continuous} is termed as the
 kernel representation with $\mathfrak{B} = \kernel{R(\sigma)}$. 
Without loss of generality, the matrix $R(\cdot)$ can be written as
$  R(\eta) \triangleq \sum_{i=L_1}^{L_2} R_{i} \eta^{i},$
where $R_{i} \in \mathbb{R}^{g \times \dimW} \forall i$ and $L_1, L_2 \in \mathbb{Z}$
with $L_1 \leq L_2$. 
Note that the Laurent polynomial ring is the natural choice for the operator
ring for behaviors with $\mathcal{T} = \mathbb{Z}$ \cite{willems1991paradigms}.


Two behaviors $\mathfrak{B}_1$ and
$\mathfrak{B}_2$ are said to be equivalent if
$\mathfrak{B}_1 = \mathfrak{B}_2$. In the kernel representation, two behaviors
are equivalent if and only if
$  \{w : R_1(\sigma)w = 0\} = \{w: R_2(\sigma)w = 0\} \iff R_1(\eta) = U(\eta)R_2(\eta),$
for some unimodular polynomial matrix $U(\eta) \in \mat{g}{g}$
\cite{willems1997introduction}. Specifically, a square polynomial matrix $U(\eta) \in \mat{g}{g}$ is said to
be unimodular if $\det{U(\eta)} \in \mathbb{R} \backslash \{0\}$ (i.e., the determinant is a
non-zero constant).

A behavior $\mathfrak{B}$ is said to be \emph{controllable} if there
exists an equivalent image representation \cite{fagnani1995image},
such that
$  w(k) = M(\sigma) \ell(k) \forall k \in \mathbb{Z}, $
where $M \in \mat{\dimW}{\dimL}$; $\ell \in \sig{\dimL}$ is termed as the latent
variable. In other words $\mathfrak{B} = \img{M(\sigma)}$.

Consider the symmetric polynomial matrix $\phi(\xi,\eta) \in \bilinearmat{q_1}{q_2}$ in commuting indeterminates $\xi$ and $\eta$ as
\begin{equation}
  \phi(\xi,\eta) = \sum_{r,s \in \mathbb{Z}} \phi_{rs} \xi^r \eta^s, 
  \label{eq:qdf_matrix_defn}
\end{equation}
where $\phi_{rs} \in \mathbb{R}^{q_1 \times q_2}$ for $r,s \in \mathbb{Z}$.
We assume that the sum in \eqref{eq:qdf_matrix_defn} exists and is finite.
The matrix $\phi$ induces a bilinear form
$L_\phi : \sig{q_1} \times \sig{q_2} \to \sig{}$ defined by
\begin{equation}
  (L_\phi(w_1,w_2))(k) \triangleq \sum_{r,s \in \mathbb{Z}} w_1(k+r)^{\mathrm{T}} \phi_{rs} w_2(k+s),
\end{equation}
for all $k \in \mathbb{Z}$.
Subsequently, we also define the corresponding quadratic form
$Q_\phi: \sig{\dimW} \to \sig{}$ as
$  Q_\phi(w) \triangleq L_\phi(w,w).$
The QDF $Q_\phi $ is said to be non-negative (denoted by $Q_\phi \geq 0 $) if
$(Q_\phi(w))(k) \geq 0 \forall k \in \mathbb{Z}$.

The trajectory set $\sig{\numIneq}$ is a real vector space, and we denote by
$\inprod{\cdot}{\cdot}$ the map 
\begin{equation}
  \label{eq:inprod_defn}
  y_1, y_2 \in \sig{\numIneq} \mapsto \inprod{y_1}{y_2} \triangleq \sum_{k \in \mathbb{Z}} y_1(k)^{\mathrm{T}} y_2(k).
\end{equation}
The infinite sum in \eqref{eq:inprod_defn} may not be well defined for
arbitrary $y_1, y_2 \in \sig{\numIneq}$; we therefore restrict our attention to
the subset $\mathcal{Y} \subset \sig{\numIneq}$ for which it is well-defined. In
particular, for finite-support trajectories the sum is finite and
$\inprod{\cdot }{\cdot }$ defines an inner product for the pre-Hilbert space
$\mathcal{Y}$. The induced norm $\|y\| \triangleq \sqrt{\inprod{y}{y}} \forall
y \in \mathcal{Y}$ and its completion with respect to $\| \cdot  \|_{\inprod{\cdot}{\cdot}}$
then yield a Hilbert space.
 
Using this definition of inner product we can define,
without loss of generality, the pointed convex cone   
$\mathcal{P} \triangleq \{w \in \sig{\dimW} | w(k) \geq 0 \forall k \in \mathbb{Z}\}$,
which would be used in the later sections for the definition of convex
constraint sets. It can be shown that the cone $\mathcal{P}$ satisfies 
\begin{equation}
  \label{eq:cone_restriction}
  \inprod{x}{y} \geq 0, 
  \quad \forall x, y \in \mathcal{P}.
\end{equation}
Such cones are said to be \emph{self-dual} with respect to the inner product,
and this property will be used in the proof of Theorem~\ref{thm:feasibility}.
Using the convex cone $\mathcal{P}$, we can define the ordering 
``$\leqp$'' by
$x_1 \leqp x_2 \iff  x_2 - x_1 \in \mathcal{P}$
\cite{luenberger1997optimization}. We term the convex cone $\mathcal{P}$ as the \textit{positive cone}. The pointedness assumption (i.e., $\mathcal{P} \cap (-\mathcal{P}) = \{0\}$) ensures that the induced relation defines a proper partial order.

\section{Constraints in the Behavioral approach}%
\label{sec:constraints}

Inspired by the developments in \cite{ten1997unilaterally}, we now introduce a general class of behavioral
models defined by inequalities. 

\begin{dfn}[Behavioral Inequalities]
  \label{defn:behavioral_ineq}
Consider the universum $\mathbb{U} \triangleq \mathcal{W}^{\mathcal{T}}$ for the index set $\mathcal{T}$ and the
alphabet set $\mathcal{W}$, a vector space $\mathbb{E}$ and
$f: \mathbb{U}\to \mathbb{E}$. The behavioral model $(\mathbb{U},\mathfrak{B})$
with $\mathfrak{B} \triangleq \{w \in \mathbb{U} :f(w) \leqp 0 \}$ is called a
behavioral inequality, where the ordering ``$\leqp$'' is with respect to the
positive cone $\mathcal{P} \subseteq \mathbb{E}$.
\end{dfn}

Definition \ref{defn:behavioral_ineq} admits a wide variety of mathematical
models of practical significance. See
\cite{ten1997unilaterally,padoan2023data,bandyopadhyay2025behavioral} for some
examples.
In general, we may write a system of temporal inequalities with constant coefficients as
\begin{equation}
  \label{eq:general_ineq}
 H(\sigma)w \leqp g,
\end{equation}
where $w \in \sig{\dimW}$, $H\in \mat{\numIneq}{\dimW}$ and
$g \in \sig{\numIneq}$. We now define the
behavioral inequality constraints $\Bcons$ as
\begin{equation}
  \label{eq:behav_defn}
 \Bcons \triangleq \{w \in \sig{\dimW} | H(\sigma)w \leqp g \}.
\end{equation}
In other words, a trajectory $w$ is said to be a solution (if any) to the behavioral inequality
\eqref{eq:general_ineq} if the transformation $k = H(\sigma)w$ satisfies
$k \leqp g$ (i.e., $g - k \in \mathcal{P}$) \cite{bandyopadhyay2025behavioral}.

\subsection{Feasibility of Behavioral Inequalities}%

We begin by characterizing when a behavioral inequality admits a non-empty set
of trajectories.
To this end, we employ a generalized theorem of alternatives for
infinite-dimensional spaces, namely the Generalized Farkas’ Lemma
\cite{clark2006necessary}, adapted here to the behavioral setting.

\begin{thm}[Feasibility]
  \label{thm:feasibility}
 Given the behavioral inequality defined in \eqref{eq:behav_defn}
 , define the
 corresponding adjoint behavior
\begin{equation*}
  \Badj \triangleq \{y \in \sig{\numIneq} | H^{\mathrm{T}}(\sigma^{-1})y = 0  \} = \kernel{H^{\mathrm{T}}(\sigma^{-1})}.
\end{equation*}
Then
\textcolor{highlight}{the behavioral inequality constraint set $\mathfrak{B}_{\mathrm{cons}}$ is non-empty ($\mathfrak{B}_{\mathrm{cons}} \neq \emptyset$)}
if and only if
there does not exist a trajectory $y \in  \Badj$ such that
$
y \geqp 0 \quad \text{and} \quad
\inprod{y}{g} < 0.
$
In other words,
$  \Bcons \ne \emptyset
  \iff  
  \inprod{y}{g} \geq 0 \forall y \in \Badj \cap \mathcal{P}.$

\begin{proof}
The proof consists of two parts.

\begin{enumerate}
\item \emph{Necessity} 

        Assume that $\Bcons \ne \emptyset$. Thus there exists $w \in \sig{\dimW}$ such that
        $
        H(\sigma) w \leqp g.
        $
        Let $y \in \Badj \triangleq \kernel{H^{\mathrm{T}}(\sigma^{-1})}$  such that $y \geq 0$. Using \eqref{eq:cone_restriction}
        and taking the inner product yields 
        $
            \inprod{y}{H(\sigma)w} \leq \inprod{y}{g}.
        $  
        Since $y \in \kernel{H^{\mathrm{T}}(\sigma^{-1})} = \kernel{H^*(\sigma)}$, we have
        $
            \inprod{y}{H(\sigma)w} = \inprod{H^*(\sigma)y}{w} = 0,
        $  
        which implies $\inprod{y}{g} \geq 0$. Thus, there cannot exist
        $y \in  \Badj$ with $y \geqp 0$  and $\inprod{y}{g} < 0$.

  \item \emph{Sufficiency}

        Now let $\inprod{y}{g}\geq  0 \forall  y \in \Badj \cap \mathcal{P}$. Assume, for
        contradiction, $\Bcons = \emptyset$. Thus $\forall w \in  \sig{\dimW}$,
        $
        H(\sigma) w \succ_{\mathcal{P}} g.
        $
        Taking inner product with $y \forall y \in \Badj \cap \mathcal{P}$ and
        utilizing \eqref{eq:cone_restriction}, we have
        $
        \inprod{y}{H(\sigma)w} > \inprod{y}{g},
        $
        which can be rewritten using the adjoint operator as
        $
        \inprod{H^{*}(\sigma)y}{w} = 0 > \inprod{y}{g},
        $
        which contradicts $\inprod{y}{g}\geq 0$. Thus $\Bcons \neq \emptyset$.
\end{enumerate}
\end{proof}
\end{thm}
Using Theorem \ref{thm:feasibility}, we can check for the existence of solutions
for a given behavioral inequality. For the want of space, we refer the readers to the
prequel \cite{bandyopadhyay2025behavioral} for worked out examples.

\begin{rmark}
Contrary to the conjecture in \cite{ten2004farkas}, the proof of Theorem
\ref{thm:feasibility} doesn't rely on the invertibility of operators in $\mat{\numIneq}{\dimW}$,
but instead exploits the adjoint operator, which lies in the Laurent
polynomial ring.
\end{rmark}




\begin{rmark}
When $H(\cdot) \in \mat{\numIneq}{\dimW}$ is a fat matrix ($\numIneq < \dimW$),
feasibility can often be verified more easily using Theorem
\ref{thm:feasibility}, since the dual condition involves the adjoint behavior
$\Badj = \kernel{H^{\mathrm{T}}(\sigma^{-1})}$, which is then finite
dimensional. The feasibility check can thus be carried out over a
finite-dimensional space, typically reducing the complexity relative to
working directly with the primal inequality description.
\end{rmark}

\section{Constrained Optimal Control using Behavioral approach}%
\label{sec:lq}

We now extend the classical Linear Quadratic optimal control framework to the
behavioral context, wherein the constraints and objectives are imposed on
trajectories rather than partitioned as state, input or output variables.
Consider an LTI system described in the minimal kernel representation as

\begin{equation}
  \Bsys = \{w \in \sig{\dimW} :R(\sigma) w = 0 \},
\label{eq:system}
\end{equation}
where, $R \in \mat{\numEq}{\dimW}$. The admissible trajectories are further
restricted by behavioral inequality constraints of the form
$\Bcons \triangleq \{ w \in \sig{\dimW} :H(\sigma)w \leqp g\}$,
where $H \in \mat{\numIneq}{\dimW}$, $g \in \sig{\numIneq}$,
and the inequality is with respect to a convex cone
$\mathcal{P} \subseteq \sig{\numIneq}$ with non-empty interior. We assume that the
constrained set of trajectories has a non-empty interior. 

\noindent We encode the initial conditions for the constrained optimal control problem as
$  I(\sigma) w(0) = a, $
where $I \in \mat{\numInit}{\dimW}$ and $a \in \mathbb{R}^{\numInit}$. Without loss
of generality, we assume that the rows of $I$ are linearly independent over
$\mathbb{R}$, and that the initial conditions admit a set of feasible
trajectories satisfying the system dynamics. 

Let $\phi(\xi,\eta) \in \bilinearmat{\dimW}{\dimW}$ be a user-defined symmetric
matrix such that the induced QDF $Q_\phi(w)$ is non-negative. We now formulate the
constrained optimal control problem as follows

\begin{prb}
  \label{prob:ProblemMain}
  \begin{subequations}
    \label{eq:behavioral_problem}
\begin{align}
\min_{w \in \sig{\dimW}} \hspace{9pt}& \sum_{k \in \mathbb{Z}} Q_{\phi}(w(k))  \label{eq:cost_functional}, \\
  \text{subject to} \hspace{10pt} & w \in \Bsys \cap \Bcons, \\
  & I(\sigma) w(0) = a.
\end{align}
\end{subequations}
\end{prb}

To proceed, we assume that the system in \eqref{eq:system} is controllable,
i.e., there exists
$M \in \mat{\dimW}{\dimL}$ such that \cite{willems2008representations}
$w = M(\sigma) \ell$,
for some latent variable
$\ell \in \sig{\dimL}$. Substituting into \eqref{eq:behavioral_problem}, the problem reduces to
\begin{subequations}
  \label{eq:image_problem}
\begin{align}
\min_{\ell \in \sig{\dimL}} \hspace{8pt} & J(\ell) \triangleq \sum_{k \in  \mathbb{Z}} Q_{\phi'}(\ell(k)), \label{eq:objective_image}\\
  \text{subject to} \hspace{8pt} & H'(\sigma)\ell \leqp g , \label{eq:constraint_image} \\
  & I'(\sigma) \ell(0) = a, \label{eq:initial_cond_image}
\end{align}
\end{subequations}
where the QDF $Q_{\phi'}(\cdot)$ is defined using the polynomial matrix
$\phi' \in \bilinearmat{\dimL}{\dimL}$ $\triangleq $ $ M(\xi)^{\mathrm{T}}\phi(\xi,\eta) M(\eta)$,
$H'(\eta) \in \mat{\numIneq}{\dimL}  \triangleq H(\eta)M(\eta)$ and
$I'(\eta) \in \mat{\numInit}{\dimL} \triangleq I(\eta)M(\eta)$. It can be shown that
the polynomial matrix $\phi'$ is symmetric and the induced QDF $Q_{\phi'}(\cdot)$ is
non-negative.


We now provide a necessary and sufficient condition for the optimality of the
constrained optimal control in the behavioral approach.


\begin{thm}[Optimality]
  \label{thm:optimality}
  Consider the constrained optimal control problem in \eqref{eq:image_problem},
  and suppose the following assumptions hold
  \begin{enumerate}
    \item The induced QDF $Q_{\phi'}$ is non-negative, i.e.,
          $Q_{\phi'}(\ell) \geq 0 \forall \ell \in \sig{\dimL}$.
    \item The feasible set
          $\mathcal{F} \triangleq \{\ell \in \sig{\dimL} | H'(\sigma)\ell \leqp g, I'(\sigma)\ell(0) = a\}$ 
          has a non-empty relative interior.
  \end{enumerate}
  Then a trajectory $\ell^* \in \mathcal{F}$ is optimal for Problem
  \ref{prob:ProblemMain}, iff there exists a multiplier trajectory
  $\lambda \in \sig{\dimLambda}, \; \lambda \geqp 0$ such that the following conditions hold
        \begin{equation}
        2 \phi'(\sigma^{-1},\sigma) \ell^* + H'(\sigma^{-1}) ^{\mathrm{T}} \lambda = 0, \label{eq:stationarity} 
        \end{equation}
\begin{equation}
  \inprod{\lambda }{H'(\sigma) \ell^* - g} = 0.
  \label{eq:complementarity}
\end{equation}
\end{thm}

\begin{proof}
  It can be shown that the problem in \eqref{eq:image_problem} is an 
 infinite dimensional convex optimization problem \cite{luenberger1997optimization}.
  We define the Lagrangian as
  $\mathcal{L}(\ell, \lambda) \triangleq J(\ell) + \inprod{\lambda }{H'(\sigma)\ell - g}$.
  where  $\lambda \in \sig{\numIneq}, \lambda \geqp 0$.
  It can be shown that the G\^{a}teaux differential of $\mathcal{L}(\cdot)$ at $\ell^*$
  for arbitrary $\delta \ell \in \sig{\dimL}$ is 
\begin{equation}
  \delta \mathcal{L}(\cdot) = \inprod{2\phi'(\sigma^{-1},\sigma)\ell^*}{\delta \ell} + \inprod{H'(\sigma^{-1})^{\mathrm{T}} \lambda }{\delta \ell}.
  \label{eq:partial_Lagrangian_behav}
\end{equation}
  Applying the Generalized Karush-Kuhn-Tucker Theorem \cite[Theorem
  9.4.1]{luenberger1997optimization}, the optimality of $\ell^*$ is equivalent
  to the existence of $\lambda \geqp 0$ such that 
$\delta \mathcal{L}(\ell^*,\lambda; \delta \ell) = 0$,
$\inprod{\lambda}{H'(\sigma)\ell^* - g} = 0$.
Using \eqref{eq:partial_Lagrangian_behav} and letting 
 $\delta \ell$ be arbitrary, we have
 $
        2 \phi'(\sigma^{-1},\sigma) \ell^* + H'(\sigma^{-1}) ^{\mathrm{T}} \lambda = 0,
$
which completes the proof.
\end{proof}


The stationarity condition in \eqref{eq:stationarity}, admits a
kernel representation of an augmented system with the behavior
$\Baug \triangleq \kernel{R_{\text{aug}}(\sigma)}$,
where, $R_{\text{aug}}(\eta) \in \mat{\dimL}{(\dimL+\numIneq)}$ defined as
$R_{\text{aug}}(\eta) \triangleq \begin{bmatrix}
    2\phi'(\eta^{-1},\eta) & H'(\eta^{-1})^{\mathrm{T}}
  \end{bmatrix}.$

If $R_{\text{aug}}(\eta)$ is full row rank, then there exist subbehaviors
$\Baut$ and $\Bcontr$ of
$\Baug$ such that
$  \Baug = \Baut \oplus \Bcontr,$
where $\Baut$ is autonomous and
$\Bcontr$ is controllable
\cite[Theorem 5.2.14]{willems1997introduction}, and the controllable
subbehavior is unique (cf. the Smith-Normal form based decomposition in
\cite{willems1997introduction}). Without loss of generality, we can write the subbehaviors $\Baut$ and $\Bcontr$
in the kernel representation as 
   $\Baut = \kernel{R_{\text{aut}}(\sigma)}, \quad
   \Bcontr = \kernel{R_{\text{contr}}(\sigma)} $
where
$R_{\text{aut}}, R_{\text{contr}} \in \mat{\bullet}{(\dimL+\dimLambda)}$ are
polynomial matrices of appropriate dimensions. In the special case where $\Baug$ is
controllable, we have $\Baug = \Bcontr$ and $\Baut = \{0\}$.

We now study the subbehaviors individually. The controllable subbehavior $\Bcontr$ admits an image
representation with the driving variable $v \in \sig{\dimV}$ such that,
$\Bcontr = \img{N(\sigma)},$
for some polynomial matrix $N(\eta) \in \mat{(\dimL+\dimLambda)}{\dimV}$. 
We decompose $N(\eta)$ into top and bottom block rows corresponding to the latent variable $\ell$ and dual multiplier variable $\lambda$, respectively:
\begin{equation}
\begin{aligned}
  \ell_{\text{contr}}(k) &\triangleq N_1(\sigma) v(k), \\
  \lambda_{\text{contr}}(k) &\triangleq N_2(\sigma) v(k) \forall k \in  \mathbb{Z}, 
\end{aligned}
\end{equation}
where, $N_1(\eta) \in \mat{\dimL}{\dimV}$ and
$N_2(\eta) \in \mat{\dimLambda}{\dimV}$. Additionally, the autonomous
subbehavior $\Baut$ consists of trajectories of the form
\cite[Theorem 3.2.16]{willems1997introduction}
\begin{equation}
  \begin{bmatrix}
    \ell_{\text{aut}}(k) \\
    \lambda_{\text{aut}}(k)
  \end{bmatrix} = \sum_{i=1}^{N} \sum_{j=0}^{\nu_i - 1}  k^j \mu_i^k B_{ij} \triangleq \begin{bmatrix}
    \ell_0(k) \\
    \lambda_0(k)
  \end{bmatrix}
 \forall k \in \mathbb{Z},
\end{equation}
where $\mu_i \in \mathbb{C}$, $i = 1,2, \dots, N$ are the distinct roots of
$\text{det} (R_{\text{aut}}(\eta)) = 0$ with multiplicities $\nu_i$ and the
constants $B_{ij} \in \mathbb{R}^{\dimL+ \dimLambda} \forall i,j $ are
determined by the initial conditions of $\ell$ and $\lambda$.

Consequently, trajectories in $\Baug$ admit the decomposition 
\begin{equation}
  \label{eq:landlambda_param}
  \begin{aligned}
    \ell(k) &= N_1(\sigma) v(k) + \ell_0(k), \\
    \lambda(k) &= N_2(\sigma) v(k) + \lambda_0(k) \forall k \in \mathbb{Z}.
\end{aligned}
\end{equation}
Substituting this decomposition into the complementarity condition in \eqref{eq:complementarity} yields 
\begin{equation}
  \label{eq:big_complementary}
\inprod{N_2(\sigma)v + \lambda_0}{H'(\sigma) N_1(\sigma)v + H'(\sigma)\ell_0 - g} = 0,
\end{equation}
with the dual variable satisfying 
$N_2(\sigma) v + \lambda_0 \geqp 0.$
Without loss of generality we can rewrite the above inequality as
\begin{equation}
  \label{eq:slack_variable_equation}
 \begin{bmatrix}
 N_2(\sigma) &  - \mathbb{I}_{\dimLambda}
 \end{bmatrix}
  \begin{bmatrix}
    v \\
    s
  \end{bmatrix}
  = -\lambda_0, \quad s \geqp 0,
\end{equation}
where $s \in \sig{\dimLambda}$ and
$\mathbb{I}_{\dimLambda}$ is the identity matrix of size
$\dimLambda \times \dimLambda$. This parameterization simplifies the
complementarity condition in terms of the free variable $v$ and the non-negative
slack variable $s$; consequently, \eqref{eq:big_complementary} becomes quadratic
in the decision variable and may be solved online using quadratic solvers to
yield the optimal control solution. The equations \eqref{eq:big_complementary}
and \eqref{eq:slack_variable_equation} are algebraic in $v$ and $s$, so for each
time step $k$ we obtain $v(k)$ and $s(k)$. Subsequently, we compute the latent
variable $\ell(k)$ and the dual variable $\lambda(k)$ using
\eqref{eq:landlambda_param}, which then yield the optimal manifest variable
$w(k) = M(\sigma)\ell(k)$.

\textcolor{highlight}{
\begin{rmark}
The numerical framework is implemented in two stages. First, SageMath
\cite{sagemath} symbolically computes the image representation $M(\sigma)$,
constructs $R_{\text{aug}}(\sigma)$, and isolates the controllable sub-behavior to
derive $N_1(\sigma)$ and $N_2(\sigma)$. The infinite-horizon problem is then
truncated to a finite horizon $k \in [0, N]$, and the shift operators
$N_1(\sigma)$, $N_2(\sigma)$, and $H'(\sigma)$ are lifted into block Toeplitz
matrices, converting
\eqref{eq:big_complementary}--\eqref{eq:slack_variable_equation} into a
finite-dimensional convex QP with $s \ge 0$, solved via CVXPY \cite{cvxpy}. 
\end{rmark}
}

\section{Example}%
\label{sec:sim}
\begin{figure*}[!t]
  \centering
  \begin{subfigure}[t]{0.48\textwidth}
    \centering
    \includegraphics[width=0.68\textwidth]{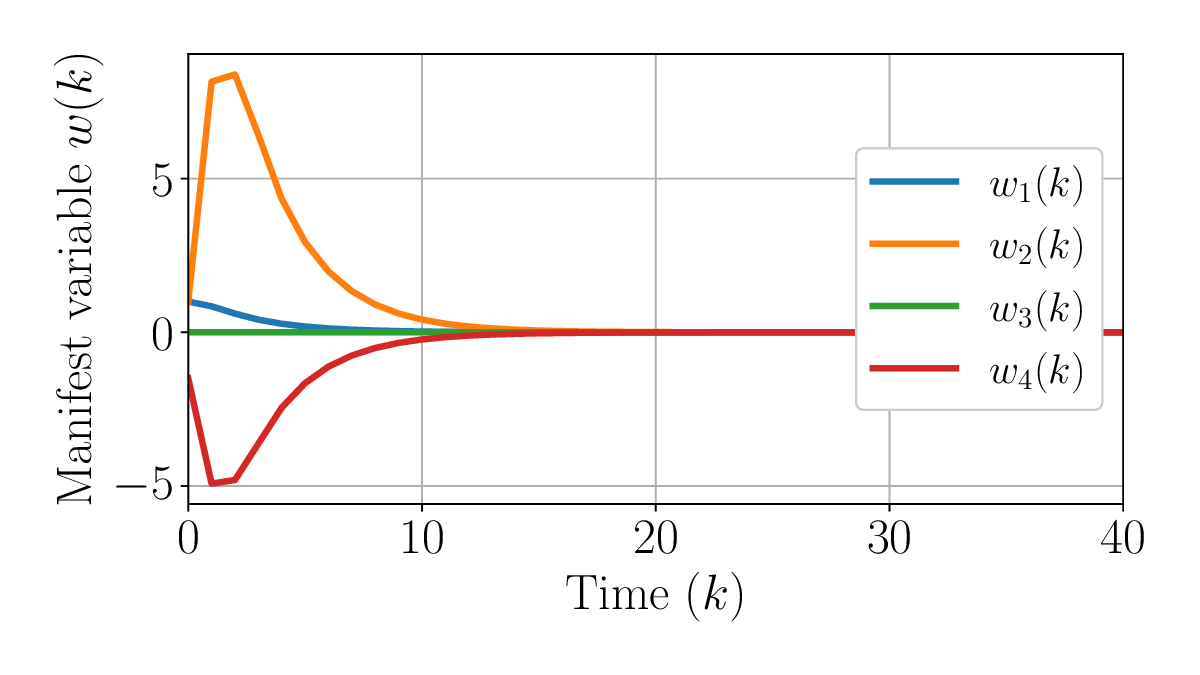}
    \caption{System trajectory of the constrained optimal control problem.}
    \label{fig:w_const}
  \end{subfigure}
  \hfill
  \begin{subfigure}[t]{0.48\textwidth}
    \centering
    \includegraphics[width=0.68\textwidth]{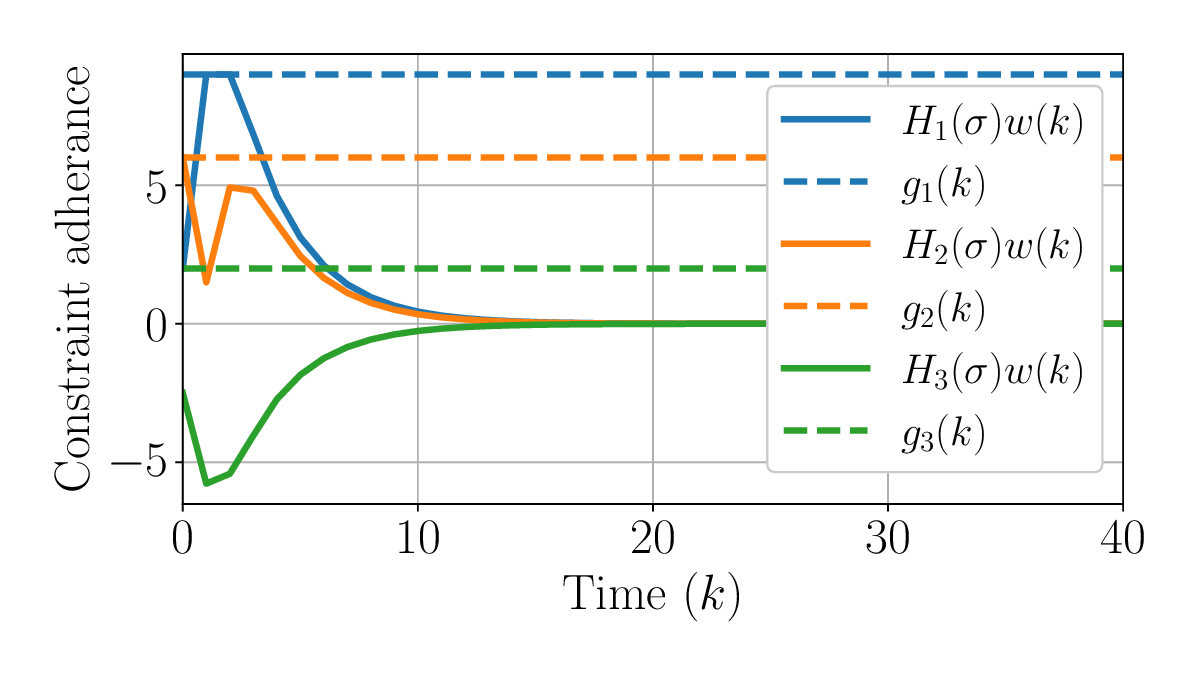}
    \caption{Inequality constraints: solid lines are the left-hand sides, dashed lines the bounds.}
    \label{fig:constraint_adherance}
  \end{subfigure}

  \begin{subfigure}[b]{0.48\textwidth}
    \centering
    \includegraphics[width=0.68\textwidth]{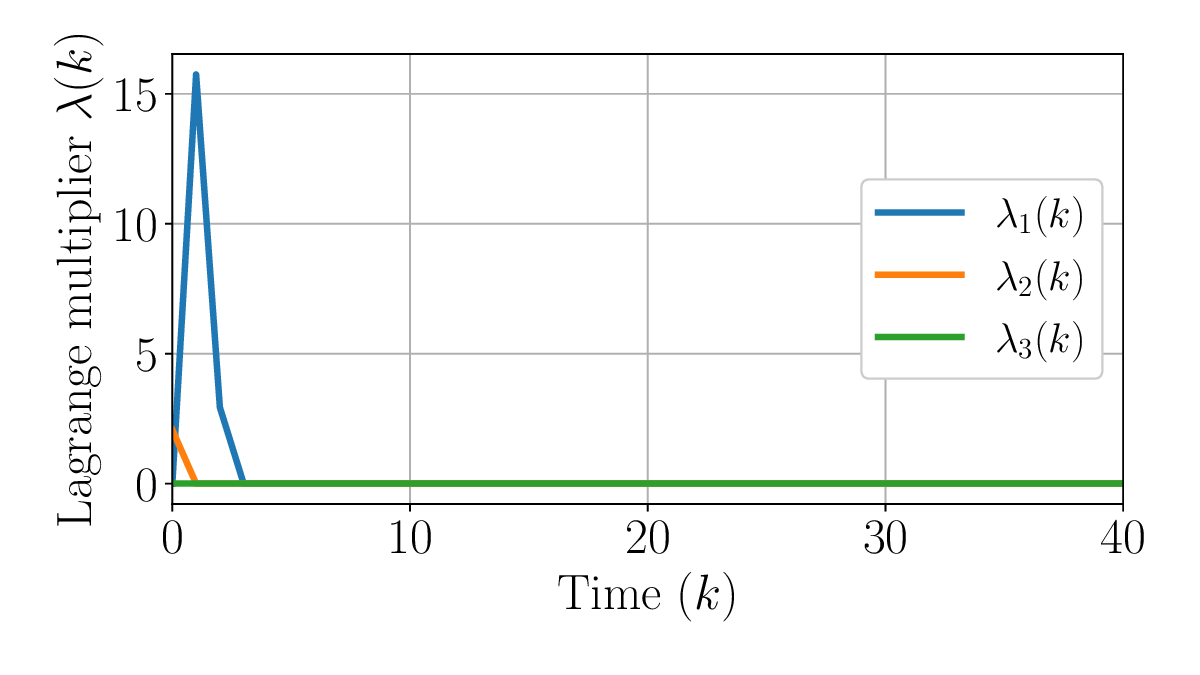}
    \caption{Optimal Lagrange multiplier trajectory of the constraints.}
    \label{fig:lambda}
  \end{subfigure}
  \hfill
  \begin{subfigure}[b]{0.48\textwidth}
    \centering
    \includegraphics[width=0.68\textwidth]{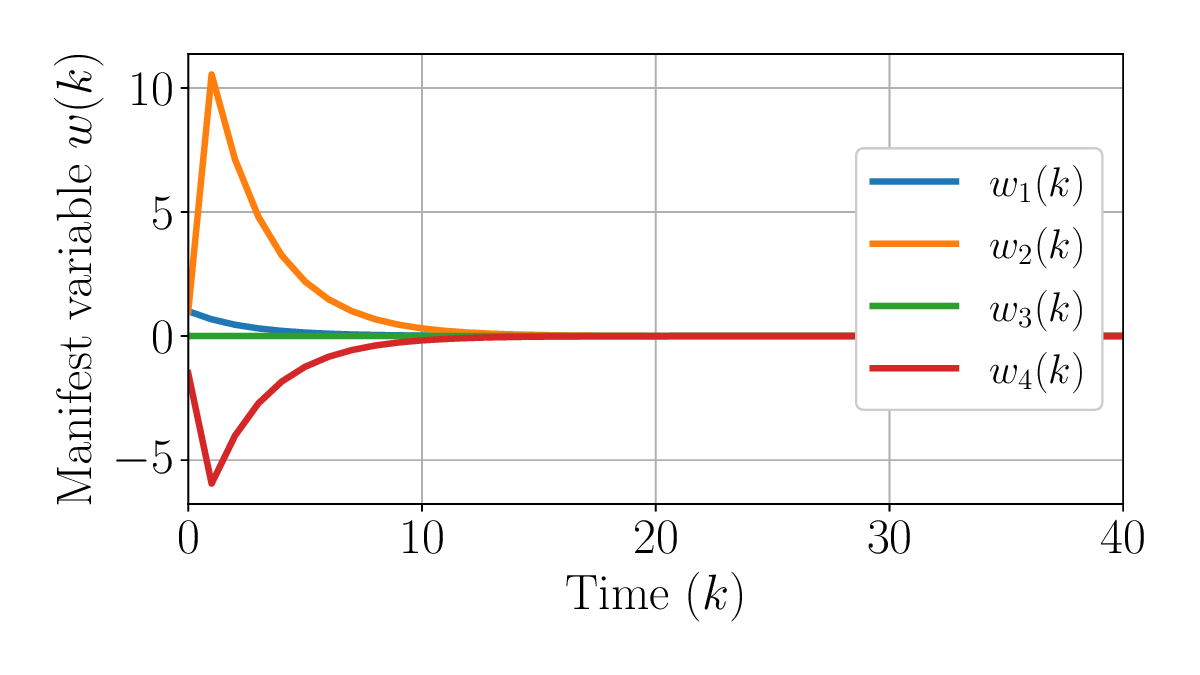}
    \caption{System trajectory under the unconstrained optimal control problem.}
    \label{fig:w_unc}
  \end{subfigure}

  \caption{Constrained Optimal Control in the Behavioral approach}
  \label{fig:overall_system_analysis}
\end{figure*}

\begin{table*}[ht]
\caption{Comparison of unconstrained and constrained costs.}
\label{tab:cost_comparison}
\centering
\rowcolors{2}{gray!15}{white} 
\renewcommand{\arraystretch}{0.78}
\resizebox{\linewidth}{!}{
  \begin{tabular}{|P{4cm}|P{4cm}|P{4cm}|P{4cm}|}
\hline
\textbf{Initial Condition ($a$)} & \textbf{Unconstrained Cost} &
                                                                 \textbf{Constrained Cost} & \textbf{Sub-optimality Gap} \\
\hline
$[1,1]^{\mathrm{T}}$ & \textbf{126.56} & 142.40 & 12.5\% \\
$[0.9,1]^{\mathrm{T}}$ & \textbf{101.06} & 104.64 & 3.5\% \\
$[0.5,1]^{\mathrm{T}}$ & \textbf{27.73} & \textbf{27.73} & 0\% \\
$[1,0.9]^{\mathrm{T}}$ & \textbf{128.19} & 143.95 & 12.3\% \\
$[1,0.5]^{\mathrm{T}}$ & \textbf{134.83} & 150.33 & 11.4\% \\
\hline
  \end{tabular}
  }
\end{table*}
Consider the linear time invariant system in the kernel representation
\begin{equation}
  \Bsys \triangleq \{w \in \sig{4} | R(\sigma)w = 0 \}, \label{eq:example_system}
\end{equation}
where $R(\eta) \in \mat{3}{4}$ is defined as
\begin{equation}
  R(\eta) \triangleq \begin{bmatrix}
\eta - 0.8 & -0.1 & 0 & -0.1 \\
-1 & -0.5 & -0.2 & -1 \\
0 & 0 & \eta - 0.7 & 0
  \end{bmatrix}.
\end{equation}
Subsequently, we impose dynamical constraints of the form
 $\Bcons \triangleq \{w \in \sig{4} | H(\sigma)w \leq g \},$
where
\begin{equation}
  H(\eta) \triangleq \begin{bmatrix}
    1 & 1 & 0 & 0 \\
    0 & 0 & 1 & -\eta \\
    -1 & 0 & 0 & 1 \\
  \end{bmatrix} \in \mat{3}{4}, g \triangleq \begin{bmatrix}
    9 \\
    6 \\
    2
  \end{bmatrix}
.
\end{equation}
The initial conditions for the dynamical system are in the form
$  I(\sigma) w(0) = a,$
where
\begin{equation}
  I(\eta) \triangleq
  \begin{bmatrix}
    1 & 0 & 0 & 0 \\
    0 & 1 & 0 & 0 \\
  \end{bmatrix} \in \mat{2}{4}, a \triangleq \begin{bmatrix}
    1 \\
    1
  \end{bmatrix}.
\end{equation}
Let the objective functional be defined as \eqref{eq:cost_functional} with
$\phi(\xi,\eta) \triangleq \mathbb{I}_{4}$, i.e., the identity matrix of dimensions $4 \times 4$. The
system \eqref{eq:example_system} can be written in the image representation
using the operator
$M(\eta) = \begin{bmatrix}
0.05 &
\eta - 0.7&
0 &
-0.5 \eta + 0.3
\end{bmatrix}^{\mathrm{T}}$
such that $w(k) = M(\sigma)\ell(k) \forall  k \in \mathbb{Z}$.

Using \eqref{eq:stationarity} and \eqref{eq:complementarity}, we obtain
the solution for a particular value of the initial conditions, implemented in
Python using SageMath \cite{sagemath} and CVXpy \cite{cvxpy}. The trajectory of the
manifest variable $w(\cdot)$ for the constrained problem is shown in
Fig. \ref{fig:w_const}, along with the constraint adherence in
Fig. \ref{fig:constraint_adherance}, where no constraint bound is ever breached.
The optimal Lagrange multiplier trajectory is shown in Fig. \ref{fig:lambda}.
For comparison, Fig. \ref{fig:w_unc} shows the system trajectory for the
unconstrained problem; the trajectories in Figs. \ref{fig:w_const} and
\ref{fig:w_unc} differ owing to the constraints imposed on the former.

We also tabulate the cost accrued by the constrained and unconstrained optimal
controllers for different values of the initial condition parameter $a$ in
Table \ref{tab:cost_comparison}.

\section{Conclusions}%
\label{sec:conclusion}
In this paper we proposed a framework for constrained linear-quadratic optimal
control within the behavioral approach, treating system dynamics and
constraints as fundamental trajectory-level objects instead of relying on
state-space or input-output representations. Such a formulation allows a more
general treatment of complex systems where the distinction between inputs and
outputs may not be fixed. We provided necessary and sufficient conditions for
the feasibility of behavioral constraints and the optimality of the constrained
optimal problem. The numerical results validated the effectiveness of the
proposed method, showing how the optimal trajectories are modified when hitting
constraints compared to the unconstrained case. While the present method
solves the open loop problem, the closed loop and data-driven extensions of the
present work may be attempted in a future work.


\bibliographystyle{ieeetr}
\bibliography{sample}

\end{document}